\documentclass[11pt,a4paper]{amsart}
\usepackage[english]{babel}
\usepackage{graphicx}

\newtheorem{theorem}{Theorem}[section]

\newtheorem{corollary}[theorem]{Corollary}

\newtheorem{proposition}[theorem]{Proposition}

\theoremstyle{definition}
\newtheorem*{definition}{Definition}

\newtheorem{example}[theorem]{Example}
\newtheorem{remark}[theorem]{Remark}

\usepackage[square,sort,comma,numbers]{natbib}

\usepackage[all]{xy}
\usepackage{pstricks}
\usepackage{enumerate}
\usepackage{amsfonts,amssymb,amsmath,eucal,pinlabel,array,hhline}
\usepackage{slashed}
\usepackage{tabulary}
\usepackage{fancyhdr}
\usepackage{a4wide}
\usepackage{calrsfs,bbm,dsfont}
\usepackage[position=b]{subcaption}

\newcommand{\Z}{\mathds{Z}}

\renewcommand{\sp}{\mathrm{sp}}
\newcommand{\eps}{\varepsilon}
 
\newenvironment{romanlist}
	{\begin{enumerate}
	}
	{\end{enumerate}}

\DeclareSymbolFont{EulerScript}{U}{eus}{m}{n}
\DeclareSymbolFontAlphabet\mathscr{EulerScript}
\begin{document}

\title{Splitting numbers and signatures}

\author{David Cimasoni}
\address{Universit\'e de Gen\`eve, Section de math\'ematiques, 2-4 rue du Li\`evre, 1211 Gen\`eve 4, Switzerland}
\email{david.cimasoni@unige.ch}
\author{Anthony Conway }
\address{Universit\'e de Gen\`eve, Section de math\'ematiques, 2-4 rue du Li\`evre, 1211 Gen\`eve 4, Switzerland}
\email{anthony.conway@unige.ch}
\author{Kleopatra Zacharova}
\address{Universit\'e de Gen\`eve, Section de math\'ematiques, 2-4 rue du Li\`evre, 1211 Gen\`eve 4, Switzerland}
\email{kleopatra.zacharova@etu.unige.ch}


\begin{abstract}
The splitting number of a link is the minimal number of crossing changes between different components required to convert it into a split link.
We obtain a lower bound on the splitting number in terms of the (multivariable) signature and nullity.
Although very elementary and easy to compute, this bound turns out to be suprisingly efficient. In particular, it makes it a routine check to recover the splitting number of 129 out of the 130 prime links with at most 9 crossings. Also, we easily determine 16 of the 17 splitting numbers that were studied by Batson and Seed using
Khovanov homology, and later computed by Cha, Friedl and Powell using a variety of techniques. Finally, we determine the splitting number of a large class of 2-bridge links
which includes examples recently computed by Borodzik and Gorsky using a Heegaard Floer theoretical criterion.
\end{abstract}

\maketitle

\section{Introduction}

Any link~$L=L_1\cup\dots\cup L_\mu$ in~$S^3$ can be turned into the split union of its components by a sequence of crossing changes between different components.
Following Batson and Seed~\cite{BS} and Cha, Friedl and Powell~\cite{CFP}, we call {\em splitting number\/} of~$L$ the minimal number of crossing changes in such
a sequence, and denote it by~$\sp(L)$. (Note that the same terminology is used in~\cite{Adams,Sh,Lackenby} for another invariant.) Obviously, upper bounds on~$\sp(L)$ can be found by inspection of diagrams, so the difficulty in computing it is to find lower bounds.

As observed in~\cite{BS} (see also~\cite[Lemma 2.1]{CFP}), the linking numbers provide an elementary such bound.
First note that~$\sp(L)$ has the same parity as the total linking number~$\sum_{i<j}\ell k(L_i,L_j)$. Furthermore, given a two component
link~$L_i\cup L_j$, let~$b_{\ell k}(L_i,L_j)$ be equal to~$0$ if~$L_i\cup L_j$ is split, to~$2$ if it is non-split but~$\ell k(L_i,L_j)$ vanishes,
and to~$|\ell k(L_i,L_j)|$ otherwise. Then, one easily shows that
\[
\sum_{i<j}b_{\ell k}(L_i,L_j)\le \sp(L)\,.
\]
We shall call this bound the {\em linking number bound\/}.

In the same article, Batson and Seed defined a spectral sequence from the Khovanov homology of a link converging to the Khovanov homology of the corresponding
split link, and used it to obtain a new lower bound on~$\sp(L)$. Testing it on links up to 12 crossings, they found 17 examples where this
{\em Batson-Seed bound\/} is strictly stronger than the linking number bound. This enabled them to compute the splitting number of 7 of these links, while the
remaining ones were left undetermined.

In~\cite{CFP}, Cha, Friedl and Powell introduced two new techniques for computing splitting numbers. The first one, based on
{\em covering link calculus\/}, relies on the following observation. If a link~$L$ has an unknotted component~$L_i$ and can be split by~$\alpha+\beta$
crossing changes,~$\alpha$ involving~$L_i$ and~$\beta$ not involving~$L_i$, then the preimage of~$L\setminus L_i$ in the~$2$-fold cover of~$L$ branched over~$L_i$
bounds a smooth surface in the~$4$-ball whose Euler characteristic can be computed from~$\alpha$,~$\beta$, and the slice genus of the components of~$L$.
(See~\cite[Theorem 3.2]{CFP} for the precise statement.) This technique turns out to be very efficient. However, it can only be used on a case by case basis, as different links
often require specific arguments (see in particular~\cite[Section 5.3]{CFP}). Furthermore, proving that a link does not bound a given surface in the~$4$-ball
(or more generally, in a rational homology~$4$-ball in case~$L_i$ is knotted) is by no means a trivial task.
Therefore the authors often need very powerful tools and heavy computations to conclude, such as results of Casson-Gordon~\cite{CG1,CG2}, the Rasmussen s-invariant~\cite{Ras},
or twisted Alexander polynomials~\cite{HKL}.

Their second result, originally stated as~\cite[Theorem 4.2]{CFP}, was then strengthened by Borodzik, Friedl and Powell in~\cite[Corollary 4.3]{BFP}
(see also~\cite[Theorem~1.1]{Kaw}). It can be stated as follows. If the multivariable Alexander polynomial~$\Delta_L$ does not vanish, then~$\sp(L)\ge \mu-1$ and if~$\sp(L)=\mu-1$, then
\[
\Delta_L(t_1,\dots,t_\mu)=\pm\prod_{i=1}^\mu\Delta_{L_i}(t_i)\cdot p(t_1,\dots,t_\mu)\cdot p(t^{-1}_1,\dots,t^{-1}_\mu)\cdot\prod_{i=1}^\mu t_i^{r_i}\cdot\prod_{i=1}^\mu(1-t_i)^{s_i}
\]
for some~$p\in\Z[t_1^{\pm 1},\dots,t_\mu^{\pm 1}]$ and~$r_i,s_i\in\Z$. This {\em Alexander polynomial obstruction\/} is powerful in showing that the splitting number of a~$\mu$-component link is strictly greater
than~$\mu-1$, but it obviously cannot be used in other cases.

These two techniques (together with the linking number bound) allowed these authors to determine the splitting numbers of the 130 prime links with up to 9 crossings. To be more precise, for 7 of these links, Cha, Friedl and Powell relied on results of Kohn~\cite{K1} on the unlinking number to exclude the value~$\sp(L)=1$; nevertheless, these values can be recovered using the Alexander polynomial obstruction of~\cite{BFP}. Furthermore, these techniques were harnessed to compute the splitting numbers of all of the 17 links in the Batson-Seed list.

Finally, in the very recent preprint~\cite{BG}, Borodzik and Gorsky found a Heegaard Floer theoretical criterion for bounding the splitting number. As an application, they showed
that for any positive~$a$, the~$2$-bridge link with Conway normal form~$C(2a,1,2a)$ has splitting number~$2a$, even though the linking number of the two components vanishes.

In the present paper, we give a new lower bound on the splitting number of a link in terms of its multivariable signature and nullity, see Theorem~\ref{thm:splitting} below.
As a special case, we obtain a bound in terms of the (one-variable) Levine-Tristram signature and nullity (Corollary \ref{cor:splitting}).
Although very elementary and extremely easy to compute (see e.g.~\cite{CC}), these bounds turn out to be remarkably powerful. Indeed, our bound is sharp for 127 out of the 130 prime
links with up to 9 crossings, and two of the remaining splitting numbers can be determined with the linking number bound. Also, our method
easily gives the splitting number of all but one of the 17 links in the Batson-Seed list.
Finally, our bound easily implies the following generalization of~\cite[Theorem~7.12]{BG}: for any~$n\ge 1$ and positive~$a_1,\dots,a_n,b_1,\dots,b_{n-1}$, the splitting number
of the~$2$-bridge link with Conway normal form~$C(2a_1,b_1,2a_2,b_2,\dots,2a_{n-1},b_{n-1},2a_n)$ is equal to~$a_1+\dots+a_n$ (Theorem~\ref{thm:2bridge}).

Let us finally mention that our results hold for the computation of the splitting number of a {\em colored link\/}. In other words, our bound holds if~$L=L_1\cup\dots\cup L_\mu$ denotes a partition of~$L$ into sublinks, and~$\sp(L)$ is the minimal number of crossing  changes between components of different sublinks to obtain the split union of these links.

\medskip

This paper is organized as follows. In Section~2, we recall the definitions of the multivariable signature and nullity, together with a list of their properties.
Section~3 contains our main result, its proof, as well as corollaries and remarks. Finally, Section~4 gathers several examples and applications.

\subsection*{Acknowledgments.} The authors wish to thank Maciej Borodzik, Jae Choon Cha, Matthias Nagel, Sebastien Ott and Mark Powell for useful discussions.
The first-named author was supported by the Swiss National Science Foundation. The second-named author was supported by the NCCR SwissMAP, funded by the Swiss National Science Foundation.

\section{Multivariable signatures}
\label{sec:sign}

Recall that the {\em Levine-Tristram signature\/} \cite{Mur,Lev,Tri} of an oriented link~$L$ is the map
\[
\sigma_L \colon S^1\setminus\{1\}\to\Z \,,
\]
where~$\sigma_L(\omega)$ is the signature of the Hermitian matrix~$H(\omega)=(1-\omega)A+(1-\overline\omega)A^T$ and~$A$
is any Seifert matrix for the oriented link~$L$. The {\em nullity\/} of~$L$ is the map~$\eta_L\colon S^1\setminus\{1\}\to\mathbb{N}$ obtained by considering the nullity
of~$H(\omega)$. These invariants were extended to colored links using~$C$-complexes in~\cite{CF} (see references therein for previous versions of these extensions).
We now briefly recall this construction.

A~$\mu$-{\em colored link\/} is an oriented link~$L$ in~$S^3$ whose components are partitioned into~$\mu$ sublinks~$L_1\cup\dots\cup L_\mu$. Thus, a~$1$-colored link is simply
an oriented link, while a~$\mu$-component $\mu$-colored link is nothing but an ordered link.
A {\em~$C$-complex\/}~\cite{Cooper} for a~$\mu$-colored link~$L=L_1\cup\dots\cup L_\mu$ is a union~$S=S_1\cup\dots\cup S_\mu$ of surfaces in~$S^3$ which is connected, and such that:
\begin{romanlist}
\item{for all~$i$,~$S_i$ is a  Seifert surface for the sublink~$L_i$;}
\item{for all~$i\neq j$,~$S_i\cap S_j$ is either empty or a union of clasps (see Figure \ref{fig:clasp});}
\item{for all~$i,j,k$ pairwise distinct,~$S_i\cap S_j\cap S_k$ is empty.}
\end{romanlist}

The existence of a~$C$-complex for an arbitrary colored link is fairly easy to establish, see~\cite[Lemma 1]{CimConway}.
Note that in the case~$\mu=1$, a~$C$-complex for the~$1$-colored link~$L$ is nothing but a connected Seifert surface for the oriented link~$L$.
Let us now define the corresponding generalization of the Seifert form.  

\begin{figure}[Htb]
\labellist\small\hair 2.5pt
\pinlabel {$x$} at 18 58
\pinlabel {$S_i$} at 94 110
\pinlabel {$S_j$} at 331 118
\endlabellist
\centering
\includegraphics[width=0.4\textwidth]{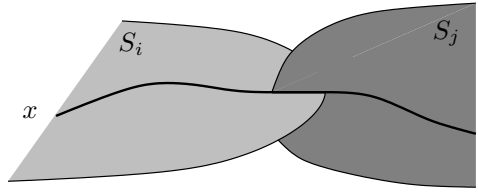}
\caption{A clasp intersection crossed by a~$1$-cycle~$x$.}
\label{fig:clasp}
\end{figure}

Given a sequence~$\eps=(\eps_1,\dots,\eps_\mu)$ of signs~$\pm 1$, let~$i^\eps\colon H_1(S)\to H_1(S^3\setminus S)$ be defined as follows. Any homology class in~$H_1(S)$ can be represented by an oriented cycle~$x$ which behaves as
illustrated in Figure~\ref{fig:clasp} whenever crossing a clasp. Then, define~$i^\eps([x])$ as
the class of the~$1$-cycle obtained by pushing~$x$ in the~$\eps_i$-normal direction off~$S_i$ for~$i=1,\dots,\mu$, and consider the bilinear form
\[
\alpha^\eps\colon H_1(S)\times H_1(S)\to\Z,\quad(x,y)\mapsto\ell k(i^\eps(x),y)\,,
\]
where~$\ell k$ denotes the linking number.
Fix a basis of~$H_1(S)$ and denote by~$A^\eps$ the matrix of~$\alpha^\eps$. Note that for all~$\eps$,~$A^{-\eps}$ is equal to~$(A^\eps)^T$. Using this fact, one easily checks that for any~$\mathbf{\omega}=(\omega_1,\dots,\omega_\mu)$ in the~$\mu$-dimensional 
torus~$\mathbb{T}^\mu$, the matrix
\[
H(\omega)=\sum_\eps\prod_{i=1}^\mu(1-\overline{\omega}_i^{\eps_i})\,A^\eps
\]
is Hermitian. Since this matrix vanishes when one of the coordinates of~$\omega$ is equal to~$1$, we restrict ourselves to the subset~$\mathbb{T}_*^\mu=(S^1\setminus\{1\})^\mu$
of~$\mathbb{T}^\mu$.

\begin{definition}
The {\em multivariable signature and nullity} of the~$\mu$-colored link~$L$ are the functions
\[
\sigma_L,\eta_L \colon\mathbb{T}_*^\mu\to\Z \,,
\]
where~$\sigma_L(\omega)$ is the signature of the Hermitian matrix~$H(\omega)$ and~$\eta_L(\omega)$ its nullity.
\end{definition}

Note that in the case~$\mu=1$, one clearly recovers the Levine-Tristram signature and nullity.
This multivariable generalization not only turns out to be well-defined (i.e. independent of the choice of the~$C$-complex), but it also satisfies
all the properties of the Levine-Tristram invariants, generalized from oriented links to colored links (see~\cite{CF} and Proposition~\ref{prop:properties} below).
Let us first illustrate this definition with an example.

\begin{example}
\label{ex:2bridge}
Figure~\ref{fig:2bridge} shows the~$2$-bridge link~$L$ with Conway normal form~$C(4,3,2)$, together with a natural~$C$-complex~$S$ for it. A natural basis for~$H_1(S)$ is given by
the cycles~$\alpha,\beta$, which are also depicted on this figure. With respect to this basis, one gets~$A^{++}=\left[\begin{smallmatrix}0&\phantom{-}0\\0&-2\end{smallmatrix}\right]$
and~$A^{+-}=\left[\begin{smallmatrix}-1&\phantom{-}1\\ \phantom{-}0&-1\end{smallmatrix}\right]$, leading in particular
to~$H(-1,-1)=4\left[\begin{smallmatrix}-2&\phantom{-}1\\\phantom{-}1&-6\end{smallmatrix}\right]$. Therefore, we have~$\sigma_L(-1,-1)=-2$.
\end{example}

\begin{figure}[Htb]
\labellist\small\hair 2.5pt
\pinlabel {$\alpha$} at 815 229
\pinlabel {$\beta$} at 1337 231
\endlabellist
\centerline{\psfig{file=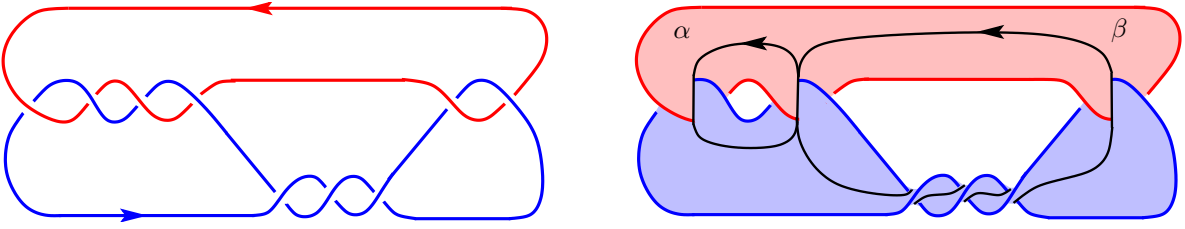,width=\textwidth}}
\caption{The link~$C(4,3,2)$, and a natural~$C$-complex for it.}
\label{fig:2bridge}
\end{figure}

Let us now collect some of the basic properties of these invariants, referring to Propositions~2.5,~2.8 and~2.13 of~\cite{CF} for the easy proofs.

\begin{proposition}
\label{prop:properties}~
\begin{romanlist}
\item Consider the~$(\mu+\mu')$-colored link~$L\sqcup L'$ given by the split union of the~$\mu$ and~$\mu'$-colored links~$L$ and~$L'$. Then, for all~$\omega\in\mathbb{T}^\mu_*$
and~$\omega'\in\mathbb{T}^{\mu'}_*$, 
\[
\sigma_{L\sqcup L'}(\omega,\omega')=\sigma_{L}(\omega)+\sigma_{L'}(\omega')\quad\text{and}\quad\eta_{L\sqcup L'}(\omega,\omega')=\eta_{L}(\omega)+\eta_{L'}(\omega') +1\,.
\]
\item Given any colored link~$L=L_1\cup\dots\cup L_\mu$, the Levine-Tristram signature and nullity of the underlying oriented link can be recovered from their multivariable analogue as
follows: for all~$\omega\in S^1\setminus\{1\}$,
\[
\sigma_L(\omega)=\sigma_L(\omega,\dots,\omega)-\sum_{i<j}\ell k(L_i,L_j)\quad\text{and}\quad\eta_L(\omega)=\eta_L(\omega,\dots,\omega)\,.
\]
\item Let~$L=L_1\cup\dots\cup L_\mu$ be a colored link, and let~$L'$ be the colored link obtained from~$L$ by reversing the orientation of every component of the sublink~$L_1$.
Then, for all~$(\omega_1,\dots,\omega_\mu)\in\mathbb{T}^\mu_*$,
\[
\sigma_{L'}(\omega_1,\dots,\omega_\mu)=\sigma_{L}(\omega^{-1}_1,\omega_2,\dots,\omega_\mu)\quad\text{and}\quad\eta_{L'}(\omega_1,\dots,\omega_\mu)=\eta_{L}(\omega^{-1}_1,\omega_2,\dots,\omega_\mu)\,.\qed
\]
\end{romanlist}
\end{proposition}

We now recall a deeper result, that requires some terminology. Let~$L=L_1 \cup \dots \cup L_\mu$ be a colored link and let~$X_L$ denote its exterior. The
epimorphism~$\pi_1(X_L) \rightarrow \mathbb{Z}^\mu$ given by~$\gamma \mapsto (\ell k(\gamma,L_1),\dots,\ell k(\gamma,L_\mu))$ induces a
regular~$\mathbb{Z}^\mu$-covering~$\widetilde{X}_L \rightarrow X_L$. The homology of~$\widetilde{X}_L$ is naturally a module over~$\Lambda=\mathbb{Z}[t_1^{\pm 1},\dots,t_\mu^{\pm 1}]$, where~$t_i$ denotes the covering transformation corresponding to an oriented meridian of~$L_i$. The~$\Lambda$-module~$H_1(\widetilde{X}_L)$ is called the {\em Alexander module\/} of the colored link~$L$. We shall denote its rank by~$\beta(L)$.

For~$r\ge 0$, let~$\Delta^{(r)}_L\in\Lambda$ denote the greatest common divisor of all~$(m-r) \times (m-r)$ minors of an~$m\times n$ presentation matrix of~$H_1(\widetilde{X}_L)$. This polynomial, well-defined up to units of~$\Lambda$, is called the~{\em $r^{\mathrm{th}}$-Alexander polynomial\/} of~$L$.
Finally,~$\Delta^{(0)}_L$ is simply called the {\em Alexander polynomial\/} of~$L$, and denoted by~$\Delta_L$. Note that the first non-vanishing Alexander polynomial
is~$\Delta^{(\beta(L))}_L$, that we shall denote by~$\Delta^{\mathrm{tor}}_L$ for obvious reasons.

\begin{proposition}~
\label{prop:nullity}
\begin{romanlist}
\item
The rank of the Alexander module of a colored link is the minimal value of its nullity:~$\beta(L)=\min \lbrace \eta_L(\omega) \ | \ \omega \in \mathbb{T}_*^\mu \rbrace$.
\item The signature is constant on the connected components of the complement in~$\mathbb{T}_*^\mu$ of the zeroes of the Alexander polynomial.
\end{romanlist}
\end{proposition}

\begin{proof}
To show the first point, let~$E_{r}(L)$ denote the ideal of~$\Lambda$ generated by the~$(m-r) \times (m-r)$ minors of an~$m\times n$ presentation matrix
of~$H_1(\widetilde{X}_L)$. Also, let~$\Sigma_r$ denote the set consisting of all~$\omega\in\mathbb{T}^\mu_*$ such that~$p(\omega)=0$ for each~$p\in E_{r-1}(L)$.
Observe that the sets~$\Sigma_r$ form a decreasing sequence. By~\cite[Theorem $4.1$]{CF},~$\Sigma_r \setminus \Sigma_{r+1}$ consists of all~$\omega$ such
that~$\eta_L(\omega)=r$. Therefore, if~$\beta$ denotes the minimal value of~$\eta_L$, we have~$\Sigma_r=\mathbb{T}^\mu_*$ for all~$r\le\beta$
and~$\Sigma_{\beta+1}\neq\mathbb{T}^\mu_*$. Hence,~$\beta$ is equal to the maximal~$r$ such that $E_{r-1}(L)=0$, which is nothing but the rank of the Alexander module.
The second point is precisely~\cite[Corollary~4.2]{CF}.
\end{proof}

\section{Main results}
\label{sec:proof}

The {\em splitting number} $\sp(L)$ of a colored link~$L=L_1\cup\dots\cup L_\mu$ is the minimal number of crossing changes between sublinks of different colors required to
turn~$L$ into the corresponding split colored link~$L_1\sqcup\dots\sqcup L_\mu$. If~$\mu$ is equal to the number of components, which is the case to keep in mind,
one recovers the splitting number discussed in the introduction. Our main result is the following inequality.

\begin{theorem}
\label{thm:splitting}
If~$L=L_1 \cup \dots \cup L_\mu$ is a colored link, then
\[
\Big|\sigma_L(\omega_1,\dots,\omega_\mu)-\sum_{i=1}^\mu \sigma_{L_i}(\omega_i)\Big|+\Big| \mu-1-\eta_L(\omega_1,\dots,\omega_\mu)+\sum_{i=1}^\mu\eta_{L_i}(\omega_i)\Big|\leq\sp(L)
\]
for all~$(\omega_1,\dots,\omega_\mu)\in\mathbb{T}_*^\mu $. 
\end{theorem}

\begin{proof}
Let~$L'$ be a colored link obtained from~$L$ by a single crossing change involving sublinks of different colors. Then, a~$C$-complex~$S'$ for~$L'$ can be obtained from a~$C$-complex~$S$ for~$L$ by adding a clasp intersection, as illustrated in Figure~\ref{fig:change}. Since~$S$ is connected, it follows that~$H_1(S')=H_1(S)\oplus\mathbb{Z}[\gamma]$ for some loop~$\gamma$ passing through the additional clasp.
With respect to this choice of bases, the resulting Hermitian matrices can be written as
\[
H'(\omega)=  
\begin{bmatrix}
H(\omega) & z \\
\overline{z}^T &  \lambda
\end{bmatrix}
\]
for some vector~$z$ and real number~$\lambda$. It easily follows that
\[
\left| \sigma_{L}(\omega)-\sigma_{L'}(\omega) \right| +\left| \eta_{L}(\omega)-\eta_{L'}(\omega) \right|=1
\]
for all~$\omega\in\mathbb{T}^\mu_*$. Consequently if~$L=L^{(0)},L^{(1)},\dots,L^{(s)}=L_1 \sqcup \dots \sqcup L_\mu $ is a splitting sequence which realizes the splitting number, then
\begin{align*}
\sp(L)&=\sum_{i=1}^{s} \left( \left| \sigma_{L^{(i-1)}}(\omega)-\sigma_{L^{(i)}}(\omega) \right|+ \left| \eta_{L^{(i-1)}}(\omega)-\eta_{L^{(i)}}(\omega) \right| \right)\\
	&\ge \left| \sigma_L(\omega)-\sigma_{L^{(s)}}(\omega) \right|+ \left| \eta_{L}(\omega)-\eta_{L^{(s)}}(\omega) \right|\,.
\end{align*}
Since~$L^{(s)}$ is the split union of its components, the first point of Proposition~\ref{prop:properties} gives
\[
\sigma_{L^{(s)}}(\omega)=\sum_{i=1}^\mu \sigma_{L_i}(\omega_i)\quad\text{and}\quad\eta_{L^{(s)}}(\omega)=\sum_{i=1}^\mu \eta_{L_i}(\omega_i)+\mu-1\,,
\]
which completes the proof.
\end{proof}

\begin{figure}[Htb]
\labellist\small\hair 2.5pt
\pinlabel {$L_i$} at -7 370
\pinlabel {$L_j$} at 158 370
\pinlabel {$L'_i$} at 315 370
\pinlabel {$L'_j$} at 493 370
\pinlabel {$S_i$} at -15 176
\pinlabel {$S_j$} at 170 176
\pinlabel {$S'_i$} at 250 176
\pinlabel {$S'_j$} at 560 176
\endlabellist
\centerline{\psfig{file=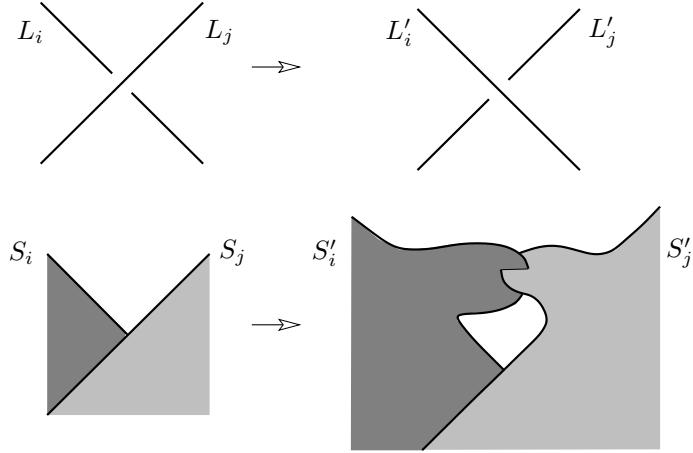,height=6cm}}
\caption{A crossing change resulting in the addition of a clasp intersection.}
\label{fig:change}
\end{figure}

\begin{remark}
The proof of Theorem~\ref{thm:splitting} has an interesting consequence. Namely, assume that~$L=L^{(0)},L^{(1)},\dots,L^{(s)}=L_1\sqcup\dots\sqcup L_\mu $ is a splitting sequence
for which there exists~$\omega\in\mathbb{T}_*^\mu$ and~$\varepsilon=\pm 1$ such that the following holds: for each~$i$,
either~$\sigma_{L^{(i+1)}}(\omega)=\sigma_{L^{(i)}}(\omega)+\varepsilon$, or~$\eta_{L^{(i+1)}}(\omega)=\eta_{L^{(i)}}(\omega)+1$. Then, this splitting sequence is minimal. This proof also shows that our bound has the same parity as the splitting number.
\end{remark}

\begin{remark}
In the case of an ordered link~$L$ and for generic~$\omega$, Theorem~\ref{thm:splitting} can also be proved by using~$4$-dimensional considerations. Indeed,
a splitting sequence of length~$s$ provides an immersed concordance in~$S^3\times[0,1]$ between~$L$ and the corresponding split link, with~$s$ transverse double points.
Removing~$B^3\times[0,1]$, where~$B^3$ is a~$3$-ball in~$S^3$ that meets each component of~$L$ in a small unknotted arc untouched by the crossing changes,
we obtain a collection of discs in~$B^4$ with~$s$ double points, bounding the link~$L\#\overline{L_1}\#\dots\#\overline{L_\mu}$. The result then follows from Theorem~7.2,
Proposition~2.10 and Proposition~2.12 of~\cite{CF}.
\end{remark}

As an immediately corollary of Theorem~\ref{thm:splitting} and the second point of Proposition~\ref{prop:properties}, we obtain the following lower bound
for~$\sp(L)$ in terms of the Levine-Tristram signature and nullity of~$L$.

\begin{corollary}
\label{cor:splitting}
If~$L=L_1 \cup \dots \cup L_\mu$ is a~$\mu$-component oriented link, then
\[
\Big|\sigma_L(\omega)+\sum_{i<j}\ell k(L_i,L_j)-\sum_{i=1}^\mu \sigma_{L_i}(\omega)\Big|+\Big|\mu-1-\eta_L(\omega)+\sum_{i=1}^\mu \eta_{L_i}(\omega)\Big|\leq\sp(L)
\]
for all~$\omega\in S^1\setminus\lbrace 1 \rbrace$.\qed
\end{corollary}

Note that there are~$2^\mu$ choices of orientations for the components of~$L$, which give~$2^{\mu-1}$ lower bounds on~$\sp(L)$.
By the third point of Proposition~\ref{prop:properties}, these correspond to the value of the multivariable invariants
on the~$2^{\mu-1}$ diagonals of~$\mathbb{T}^\mu_*\simeq(0,2\pi)^\mu$.

Our next corollary should be compared to~\cite[Corollary 4.3]{BFP}. The first inequality is identical; in view of~Proposition~\ref{prop:nullity}~{\it (ii)},
the case~$\beta(L)=0$ of the second statement can be understood as the signature analogue of the Alexander polynomial obstruction.

\begin{corollary}
\label{cor:obstruction}
Let $L=L_1 \cup \dots \cup L_\mu$ be a $\mu$-colored link,~$\beta(L)$ the rank of its Alexander module, and~$\Delta^{\mathrm{tor}}_L$ its first non-vanishing Alexander polynomial.
Then, one has
\[
\mu-1-\beta(L) \leq \sp(L)\,.
\]
Furthermore, if~$\mu-1-\beta(L)=\sp(L)$, then
\[
\sigma_L(\omega_1,\dots,\omega_\mu)= \sum_{i=1}^\mu \sigma_{L_i}(\omega_i)\quad\text{and}\quad\eta_{L_1}(\omega_1)=\dots=\eta_{L_\mu}(\omega_\mu)=0
\]
for all~$\omega=(\omega_1,\dots,\omega_\mu)\in\mathbb{T}^\mu$ such that~$\Delta^{\mathrm{tor}}_L(\omega)\neq 0$.
\end{corollary}
\begin{proof}
By Theorem~\ref{thm:splitting} and the first point of Proposition~\ref{prop:nullity}, we have the inequalities
\begin{align*}
\sp(L)&\ge\Big|\sigma_L(\omega)-\sum_{i=1}^\mu \sigma_{L_i}(\omega_i)\Big|+\Big|\mu-1-\eta_L(\omega)+\sum_{i=1}^\mu\eta_{L_i}(\omega_i)\Big|\\
	&\ge\mu-1-\eta_L(\omega)+\sum_{i=1}^\mu\eta_{L_i}(\omega_i)\\
	&\ge\mu-1-\eta_L(\omega)\\
	&\ge\mu-1-\beta(L)\,.
\end{align*}
Let us now assume that~$\Delta^{\mathrm{tor}}_L(\omega)\neq 0$. Using the notations of the proof of Proposition~\ref{prop:nullity}, this implies that~$\omega$ belongs
to~$\Sigma_{\beta(L)}\setminus\Sigma_{\beta(L)+1}$. By~\cite[Theorem~4.1]{CF}, this means that~$\eta_L(\omega)=\beta(L)$. The second statement now follows
by setting equalities for each of the inequalities displayed above.
\end{proof}

\begin{remark}
\label{rem:Ccomplex}
In the 2-component and~$\beta(L)=0$ case, one can also show Corollary~\ref{cor:obstruction} using other considerations.
Indeed, it is known that if a 2-component link has splitting number one, then it consists of a {\em band clasping\/} of two knots (see the proof of Theorem~1 in~\cite{Kondo}).
A careful analysis of the situation using~$C$-complexes then yields the result. As the Alexander polynomial may be computed via~$C$-complexes~\cite{CimConway},
this alternative method also enables one to recover the Alexander obstruction in the 2-component case (and in particular~\cite[Theorem $4.2$]{CFP}).
\end{remark}

\begin{remark}
\label{rem:splice}
As shown by Degtyarev, Florens and Lecuona~\cite{DFL}, the multivariable signature behaves in a controlled way via the operation known as {\em splicing\/}, which generalizes
the satellite, infection, and cabling operations. Therefore, Theorem~\ref{thm:splitting} together with~\cite[Theorem~2.2]{DFL} immediately yields a lower bound for the splitting
number of the splice of two links in terms of the signature of its splice components. In particular, this gives a lower bound for the splitting number of the
cabling of a link~$L$ (along one or several of its components) in terms of the signature of~$L$ and of the corresponding torus links. However, we have not been able to find compelling applications
of this result, and therefore do not discuss it in more detail.
\end{remark}

Finally note that all the techniques developed in this paragraph can also be used to obtain lower bounds on the unlinking number~$u(L)$ of a link~$L$. For instance,
if $L=L_1 \cup \dots \cup L_\mu$ is a~$\mu$-component link, then one obtains
\[
\left| \sigma_L(\omega_1,\dots,\omega_\mu)\right| +\left|\mu-1-\eta_L(\omega_1,\dots,\omega_\mu) \right|+\sum_{i<j} \left|\ell k(L_i,L_j)\right| \leq 2u(L)
\]
for all~$(\omega_1,\dots,\omega_\mu)\in\mathbb{T}^\mu_*$. Unfortunately, this bound is not very powerful so we shall not discuss it any further.

\section{Examples and applications}

In this section, we use our bound on three families of examples: the links with at most nine crossings, the links of the Batson-Seed list, and~$2$-bridge links.
All the diagrams are taken from SnapPy~\cite{SnapPy}.

\medskip

As stated above, we first tested Theorem~\ref{thm:splitting} on all 130 prime links with fewer than ten crossings, using the notations and data from LinkInfo~\cite{LinkInfo}.
In 125 cases, the Levine-Tristram bound of Corollary~\ref{cor:splitting} is enough to recover the splitting number, while Corollary~\ref{cor:obstruction} recovers two additional cases. The three remaining links are~$L9a47$ and~$L9n27$ (for which the linking number bound is sharp) and~$L8a9$ (whose splitting number can be recovered by the Alexander polynomial obstruction). Let us illustrate this with a couple of examples.

\begin{figure}[!htbp]
\centerline{\psfig{file=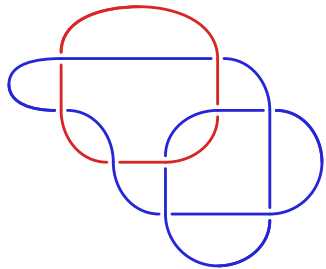,height=3.5cm}}
\caption{The link L9a29.}
\label{fig:L9a29}
\end{figure}

\begin{example}
\label{ex:L9a29}
The splitting number of the link~$L=L9a29$ depicted in Figure~\ref{fig:L9a29} was shown to be~$3$ in~\cite[Section~$4.2$]{CFP} by using the Alexander polynomial obstruction.
Orienting~$L$ so that it consists of a right-handed trefoil~$L_1$ and a trivial knot~$L_2$ with linking number~$ \ell k(L_1,L_2)=-1$, we have
\[
\sigma_L(-1)=5, \ \ \ \eta_L(-1)=0, \ \ \ \sigma_{L_1}(-1)=2, \ \ \ \eta_{L_1}(-1)=0, \ \ \ \sigma_{L_2}(-1)=0, \ \ \ \eta_{L_2}(-1)=0\,.\]
It follows that the bound
\[
\left|\sigma_L(-1)+\ell k(L_1,L_2)-\sigma_{L_1}(-1)-\sigma_{L_2}(-1)\right|+|2-1-\eta_L(-1)+\eta_{L_1}(-1)+\eta_{L_2}(-1)| =3
\]
of Corollary \ref{cor:splitting} is sharp.
\end{example}

\begin{example}
\label{ex:L9a24}
The splitting number of the link~$L=L9a24$ depicted on the left hand side of Figure~\ref{fig:L9a24maple} was shown to be~$3$ in~\cite[Section~$6$]{CFP} by using the Alexander polynomial obstruction. The lower bound of Corollary~\ref{cor:splitting} gives~$1$, so it is not sufficient to conclude. However, Corollary \ref{cor:obstruction} shows that the splitting number must be greater than one, and therefore equal to three due to the parity of the linking number. Indeed, one does not always
have~$\sigma_L(\omega_1,\omega_2)=\sigma_{L_1}(\omega_1)+\sigma_{L_2}(\omega_2)$, as shown by the zero locus of~$\Delta_L$ illustrated on the right hand side of
Figure~\ref{fig:L9a24maple}. (Recall Proposition~\ref{prop:nullity} {\it (ii)}.)

The exact same analysis shows that the splitting number of~$L9n17$ is three.
\begin{figure}[!htbp]
\captionsetup[subfigure]{position=below,justification=justified,singlelinecheck=false,labelfont=bf}
\begin{subfigure}[t]{.30\textwidth}
\psfig{file=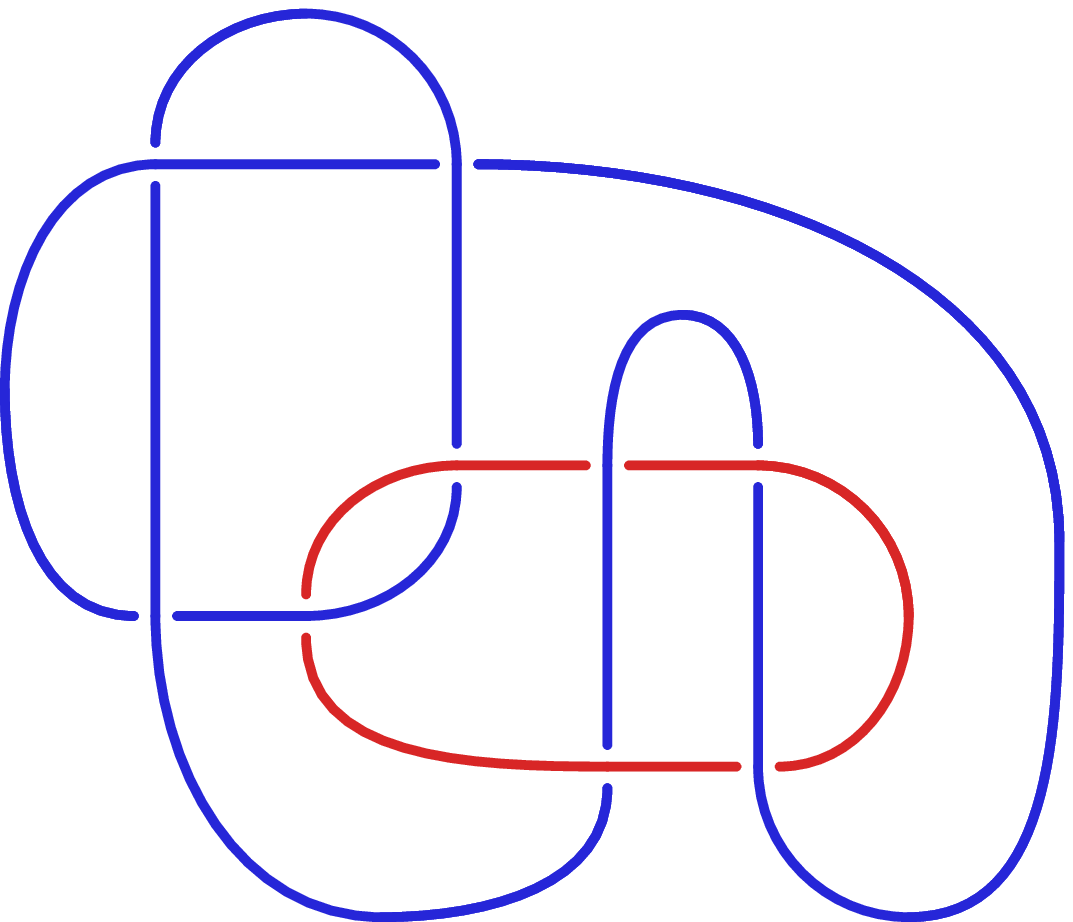,width=\linewidth}
\label{fig:L9a24}
\end{subfigure}
 \hspace{3cm}
\begin{subfigure}[t]{.30\textwidth}
\centering{\psfig{file=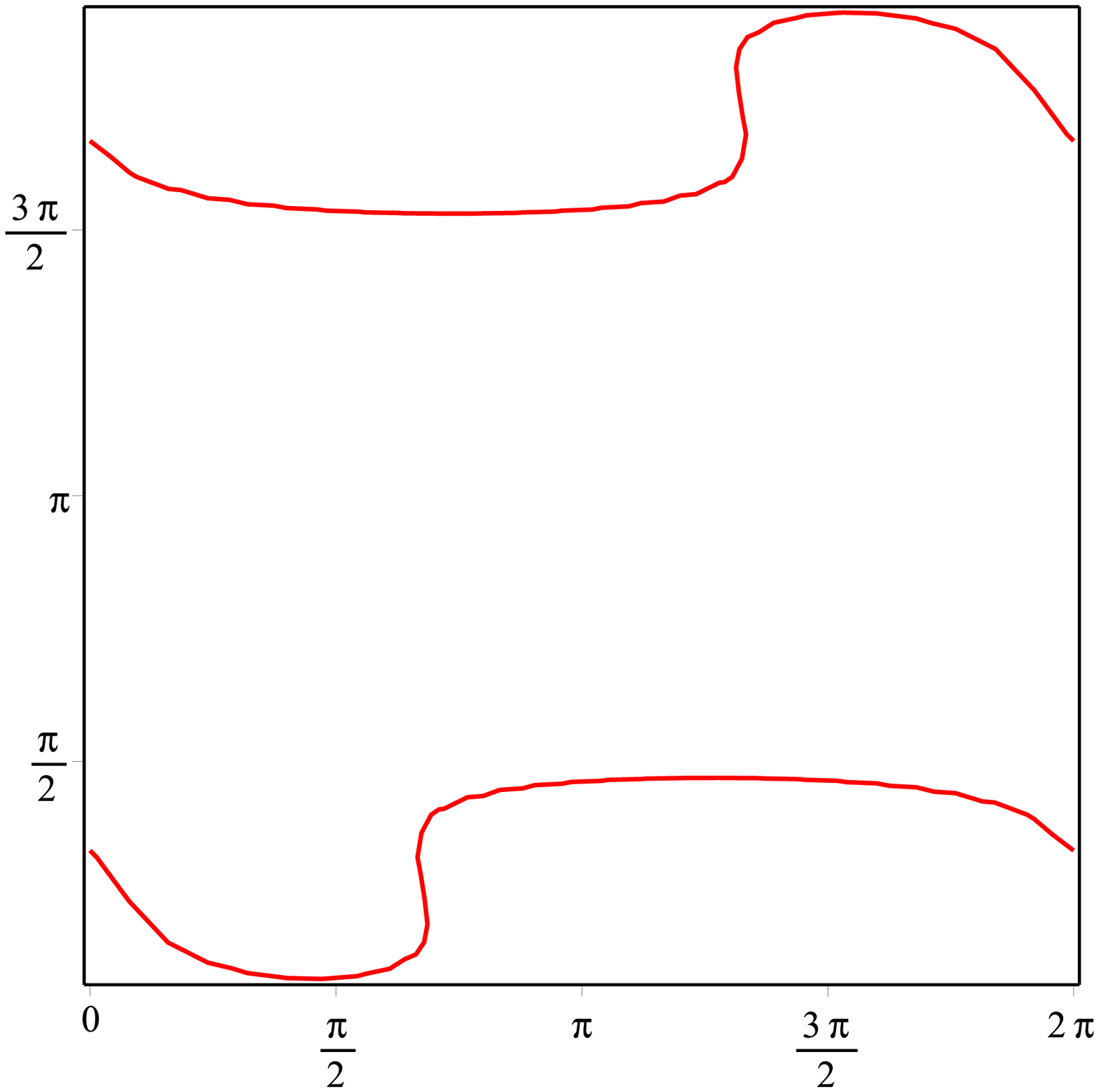,width=\linewidth}}
\label{fig:maple}
\end{subfigure}
\caption{The link $L9a24$ and the intersection of the zero locus of its Alexander polynomial with $\mathbb{T}^2_*$.}
\label{fig:L9a24maple}
\end{figure}
\end{example}

Next, we tested our bound on the 17 links of the Batson-Seed list, using Seifert matrix data kindly provided by J.C.~Cha.
Among these, there are seven~$12$-crossing links (namely,~$L12n1342, L12n1350, L12n1357, L12n1363, L12n1367, L12n1274$ and~$L12n1404$) for which both components are trefoils, and whose splitting number was shown to be equal to~$3$ by Batson and Seed. Cha, Friedl and Powell recovered these results via the Alexander obstruction. Applying Corollary~\ref{cor:obstruction} as in Example~\ref{ex:L9a24} immediately gives the same conclusion. Moreover, the splitting number of~$L12n1342, L12n1350, L12n1367$ and~$L12n1274$ can be recovered by using the Levine-Tristram signature and Corollary~\ref{cor:splitting} alone. Let us illustrate this with one example.

\begin{example}
Orient the link~$L=L12n1367$ (depicted in Figure \ref{fig:L12n1367}) so that~$\ell k(L_1,L_2)=1$ and set~$\omega=e^{\frac{i \pi}{3}}$. Using 
\[
\sigma_L(\omega)=0, \ \ \ \sigma_{L_1}(\omega)=1, \ \ \ \sigma_{L_2}(\omega)=-1, \ \ \ \eta_{L}(\omega)=\eta_{L_1}(\omega)=\eta_{L_2}(\omega)=1\,,
\]
it follows that the bound
\[
\left| \sigma_L(\omega)+\ell k(L_1,L_2)-\sigma_{L_1}(\omega)-\sigma_{L_2}(\omega)\right|+|2-1-\eta_L(\omega)+\eta_{L_1}(\omega)+\eta_{L_2}(\omega)| =3
\]
of Corollary \ref{cor:splitting} is sharp.
 \begin{figure}[!htbp]
\centerline{\psfig{file=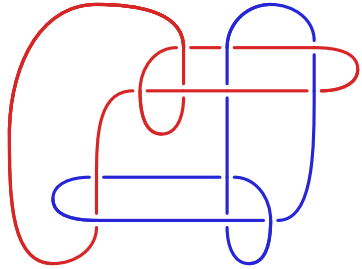,height=3.5cm}}
\caption{The link L12n1367.}
\label{fig:L12n1367}
\end{figure}
\end{example}

For the remaining 10 links in this list, Batson and Seed could not determine whether the splitting number is~$3$ or~$5$. These links having two or three components, the Alexander
polynomial obstruction cannot be applied. However, Cha, Friedl and Powell used various arguments based on covering link techniques to determine these values. As it turns out, our bound allows to easily determine these splitting number for all but one of them, namely~$L12n1321$. Here are several of these examples.

\begin{example}
Consider the link~$L=L11a372$ depicted in Figure~\ref{fig:L11a372}, whose splitting number was shown to be~$5$ in~\cite[Section~$5.2$]{CFP}.
Orient~$L$ so that its trivial components~$L_1,L_2$ satisfy~$ \ell k(L_1,L_2)=-1$. Since~$ \sigma_L(-1)=5$ and~$\eta_L(-1)=0$,
the bound 
\[
\left|\sigma_L(-1)+\ell k(L_1,L_2)\right|+|2-1-\eta_L(-1)| =5
\]
given by the classical signature is enough to conclude.
\begin{figure}[!htbp]
\centerline{\psfig{file=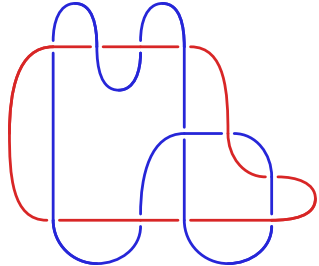,height=3.5cm}}
\caption{The link L11a372.}
\label{fig:L11a372}
\end{figure}
\end{example}

\begin{example}
\label{ex:L12a1622}
It took the whole Section~$5.3$ of~\cite{CFP} to show that the splitting number of the 3-component link~$L=L12a1622$ depicted in Figure \ref{fig:L12a1622} is equal to 5.
Orienting~$L$ so that its trivial components~$L_1,L_2,L_3$ satisfy~$\ell k(L_1,L_2)=0,  \ \ell k(L_1,L_3)=0,  \ \ell k(L_2,L_3)=1$, and
picking~$\omega=e^{\frac{3\pi i}{4}}$, we have~$\sigma_L(\omega)=-4$ and~$\eta_L(\omega)=0$. Hence, the bound 
\[
\left|\sigma_L(\omega)+\ell k(L_2,L_3) \right|+|3-1-\eta_L(\omega)| =5
\]
of Corollary \ref{cor:splitting} immediately provides the desired splitting number.
\begin{figure}[!htbp]
\centerline{\psfig{file=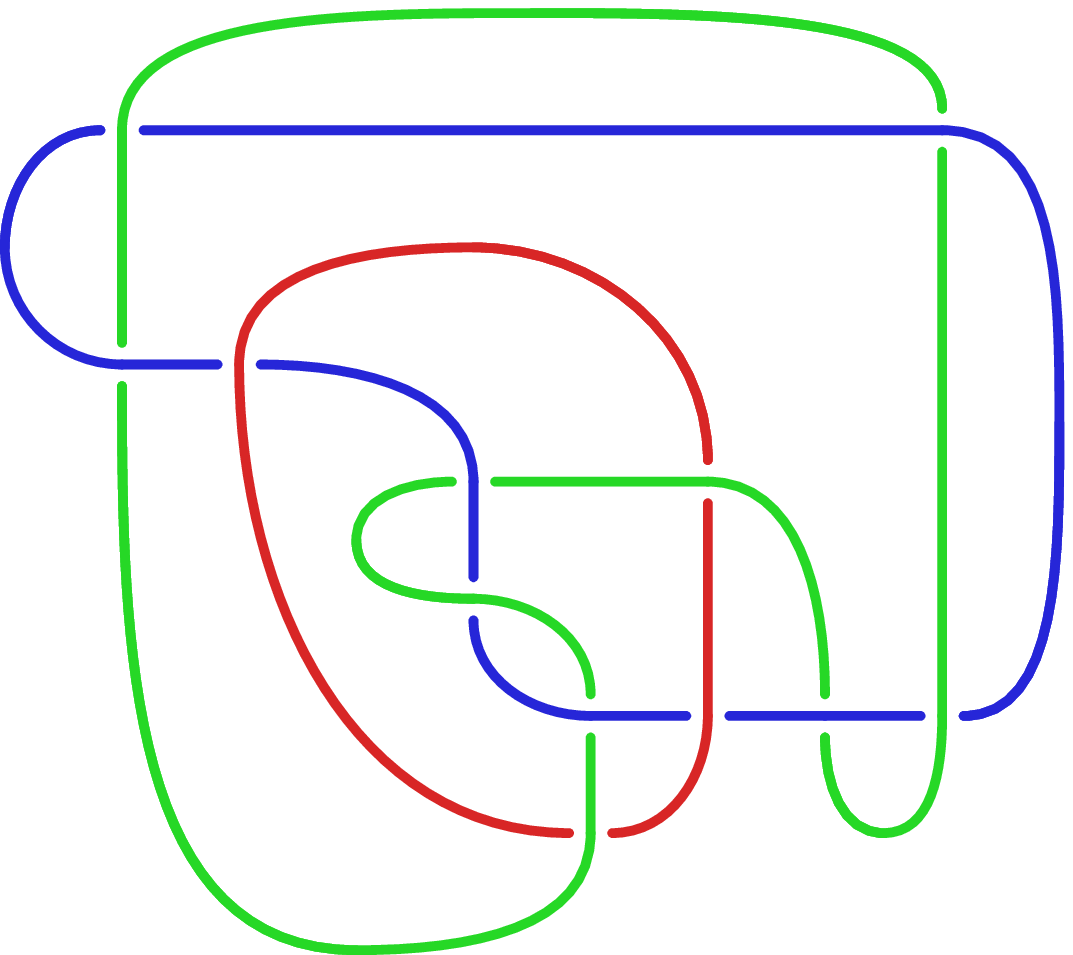,height=3.5cm}}
\caption{The link L12a1622.}
\label{fig:L12a1622}
\end{figure}
\end{example}

\begin{example}
\label{ex:L12n1326}
Consider the link~$L=L12n1326$ depicted in Figure \ref{fig:L12n1326}. Cha-Friedl-Powell~\cite[Section~$5.2$]{CFP} used the twisted Alexander polynomial to compute the slice
genus of a covering link, and concluded that~$\text{sp}(L)=3$. Orienting~$L$ so that its trivial components~$L_1,L_2$ satisfy~$ \ell k(L_1,L_2)=1$, and
picking~$\omega =e^{\frac{\pi i}{5}}$ so that~$\sigma_L(\omega)=1$ and~$\eta_L(\omega)=0$, the bound 
\[
\left|\sigma_L(\omega)+\ell k(L_1,L_2)\right|+|2-1-\eta_L(\omega)| =3
\]
of Corollary~\ref{cor:splitting} immediately provides the desired splitting number.
\begin{figure}[!htbp]
\centerline{\psfig{file=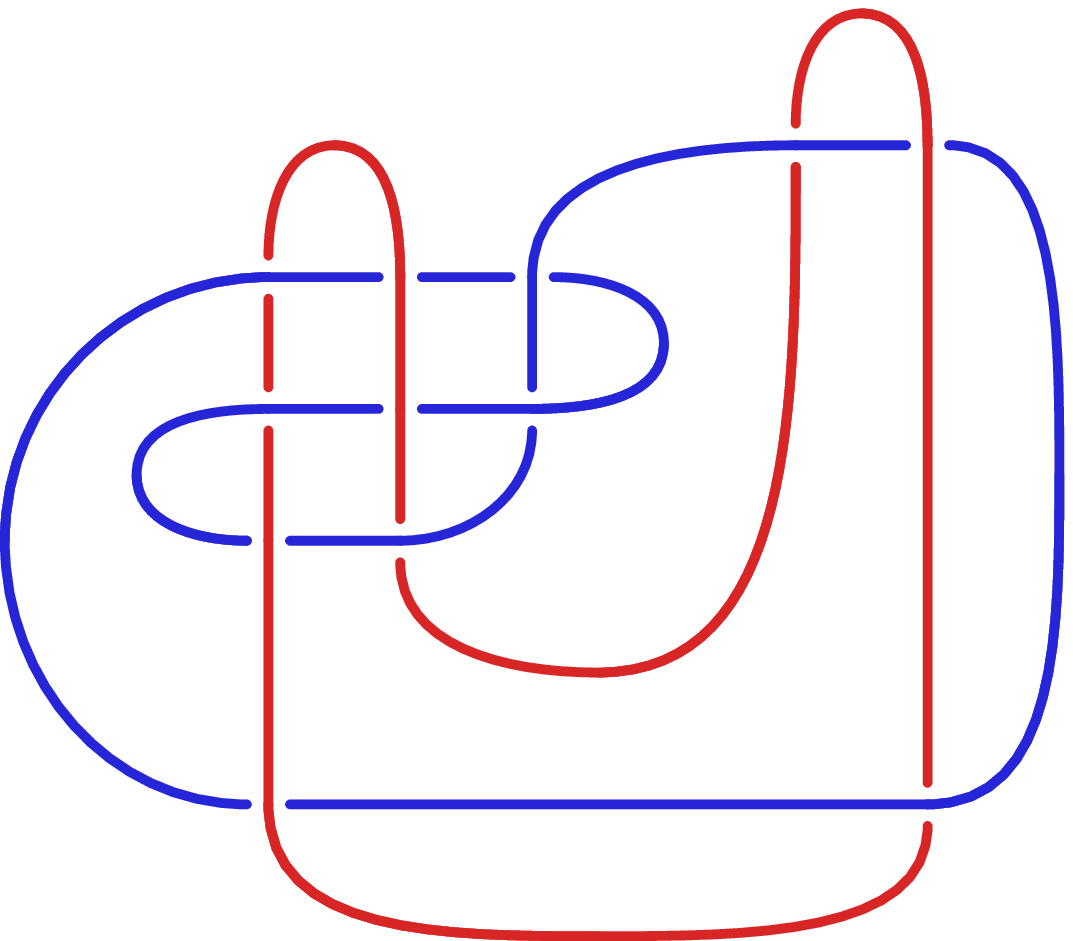,height=3.5cm}}
\caption{The link L12n1326.}
\label{fig:L12n1326}
\end{figure}
\end{example}

In the {\em very\/} recent preprint~\cite{BG}, Borodzik and Gorsky used Heegaard Floer techniques to compute the splitting number of the~$2$-bridge
link with Conway normal form~$C(2a,1,2a)$ (recall Figure~\ref{fig:2bridge} for an explanation of this notation). We conclude this article by showing how easily
we can compute the splitting number of a much wider class of~$2$-bridge links.

\begin{theorem}
\label{thm:2bridge}
For any~$n\ge 1$ and any positive integers~$a_1,\dots,a_n,b_1,\dots,b_{n-1}$, the splitting number of the~$2$-bridge link~$C(2a_1,b_1,2a_2,b_2,\dots,2a_{n-1},b_{n-1},2a_n)$
is equal to~$a_1+\dots+a_n$.
\end{theorem}
\begin{proof}
Let us denote this link by~$L$. Inspecting the standard diagram given by Conway's normal form, it is clear that~$\sp(L)\le a_1+\dots+a_n=:s$.
Furthermore, as exemplified in Figure~\ref{fig:2bridge},~$L$ admits a natural~$C$-complex~$S$ made of two discs intersecting in~$s$ clasps, yielding a natural basis
for~$H_1(S)$ of cardinality~$s-1$. Computing the associated generalized Seifert matrices leads to
\[
H(-1,-1)=4\begin{bmatrix}-2d_1&1& & & \\ 1&-2d_2&1 & & \\ &1&\ddots&\ddots \\ & & \ddots& &1\\ & & &1&-2d_{s-1}\end{bmatrix}\,,
\]
where the diagonal elements are positive integers whose precise value depends on the integers~$a_i$ and~$b_i$, but will play no role here.
Indeed, all the eigenvalues of a matrix of the form above are negative, so~$\sigma_L(-1,-1)=1-s$ and~$\eta_L(-1,-1)=0$. Since~$L$ has~$\mu=2$ components, both unknotted, the result now follows from Theorem~\ref{thm:splitting}.
\end{proof}

Note that this last result was obtained using the ordinary signature alone. Using the full power of Theorem~\ref{thm:splitting}, it is possible to determine the splitting number
of many other~$2$-bridge links, such as all~$C(a_1,\dots,a_n)$ with~$n\le 3$. The first example of a~$2$-bridge link whose splitting number is unknown to us is~$L=C(4,3,1,3)$:
the linking number vanishes and our bound reads~$\sp(L)\ge 2$, while the ``obvious'' splitting sequence has length~$4$.

\bibliographystyle{plain}
\nocite{*}
\bibliography{BiblioSplitting}

\end{document}